\documentclass{amsart}
\usepackage{amsmath,amsfonts,amsthm,enumerate,amssymb,textcomp}
\usepackage[mathscr]{euscript}
\usepackage[cp1250]{inputenc}
\usepackage[T1]{fontenc}
\def\={\discretionary{-}{-}{-}}

\newtheorem{df}{Definition}
\newtheorem{tw}[df]{Theorem}
\newtheorem{lm}[df]{Lemma}
\newtheorem{corl}{Corollary}
\newtheorem{wn}[df]{Corollary}

\newtheorem{prob}[df]{Problem}

\DeclareMathOperator\domk{cl}
\DeclareMathOperator\wne{int}
\author{Piotr Sworowski, Waldemar Sieg}
\address{Kazimierz Wielki University\endgraf Institute of Mathematics\endgraf
Powsta\'nc\'ow Wielkopolskich 2\endgraf 85-090 Bydgoszcz\endgraf Poland\endgraf e-mail: {\tt piotrus@ukw.edu.pl}
\endgraf e-mail: {\tt waldeks@ukw.edu.pl}}
\title{Uniform limits of $\mathcal B_1^{**}$-functions}
\begin{document}
\begin{abstract}
We characterise the class of uniform limits of functions from Paw\-lak's class $\mathcal B_1^{**}$. The resulting class $u\mathscr S_1$, which contains functions with the oscillation rank one, is discussed in connection with its linear span. 
We apply a general, topological space, setting in our discussion.
\end{abstract}
\maketitle
\section{Introduction}
The so-called {\em Baire-one-star} ($\mathcal B_1^*$) property of a real-valued function, defined on an interval in the real line, has received certain attention over recent couple of decades, often in connection with generalised differentiation and integration, where it is investigated if some generalised primitives or indefinite integrals nece\-ssarily obey to this, stronger than Baire-one, property; see, e.g., \cite{Gordon6}. A different kind of application reveals a work by one of the authors \cite{MalSwor2}, where the $\mathcal B_1^*$ property was used as a main tool for a characterisation of integrability analogous to the Lebesgue theorem on the Riemann integral.\footnote{Quite recently, $B_1^*$ functions received some attention in connection with dynamical systems, as so-called {\em$k$-continuous functions} or {\em right $B_1$ compositors} \cite{Ali}.}

\smallskip
Under the Baire-one-star name, and in generalised differentiation connection, the property was introduced by Richard J.\ O'Malley in \cite{OM1}, but is also known as the {\em Baire half} property after Harvey Rosen's paper \cite{Rosen}, or as {\em piecewise continuity} \cite{JayRo}. (See \cite{Kirchheim} for more info, various descriptions of $\mathcal B_1^*$ and related properties.) For the present note we can accept the following definition: an $f\colon X\to Y$, $X,Y$ topological spaces, has the $\mathcal B_1^*$ property if for every closed $D\subset X$, there is an open $O\subset X$, $D\cap O\ne\emptyset$, such that the restriction $f\restriction(D\cap O)$ is continuous (in its domain).

\smallskip
Certain limit properties of $\mathcal B_1^*$ make this class technically useful in many Baire hierarchy related problems. A general result found in Kuratowski's monograph \cite[pp.\ 294--295]{Kuratowski} says that, for an arbitary ordinal $\alpha$, any Baire class $\alpha$ function can be obtained as the uniform limit of a sequence of Baire-$\alpha$ functions all whose ranges are discrete sets. For $\alpha=1$ and in, e.g., the  $\mathbb R\to\mathbb R$ case, the discrete range property implies $\mathcal B_1^*$ property. Thus one can claim that Baire class one functions are exactly limits of uniformly convergent sequences of Baire-one-star functions. Hence, for approximation problems, the $\mathcal B_1^*$ property may be seen, in a sense, exactly as useful as $\mathcal B_1$.

\smallskip
The main objective of this note is to provide a simple characterisation for the uniform closure of the class of {\em$\mathcal B_1^{**}$-functions}, a subclass of $\mathcal B_1^*$, first considered by Ryszard J.\ Pawlak in \cite{Pawlak}. 
Since the characterisation theorem allows a simple and brief proof, we strove for presenting it in as much as possible general setting.
\section{The characterisation}
Let $(X,\mathcal T)$ be a topological space, $(Y,\rho)$ a metric space. For a function $f\colon X\to Y$ and $x\in X$ we denote with $\omega_f(x)$ the {\em oscillation} of $f$ {\em at} $x$, that is, $\omega_f(x)=\inf_O\sup_{\xi\in O}\varrho(f(x),f(\xi))$, where the inf ranges over all neighbourhoods $O\ni x$. This should not be confused with the {\em oscillation} of $f$ {\em on} a subset $Z\subset X$, which is the diameter of $f(Z)$. (Clearly, $\omega_f(x)=0$ iff $f$ is continuous at $x$.) For two functions $f,g\colon X\to Y$ we set $\varrho_\infty(f,g)=\sup_{\xi\in X}\varrho(f(\xi),g(\xi))$.\label{vbwlkviyfvwl}

\smallskip
By definition, a function $f\colon X\to Y$ is $\mathcal B_1^{**}$ if the restriction of $f$ to the set of its discontinuity points, $D_f$, if nonempty, is continuous.\footnote{In such case, notice, $D_f$ must be nowhere dense in $X$, so that $f$ has also $\mathcal B_1^*$ property.} The notation $\mathcal B_1^{**}(X,Y)$ stands for the class of all $\mathcal B_1^{**}$-functions defined over $X$ with values in $Y$. If $Y$ is the real line with the natural metric, we write instead just $\mathcal B_1^{**}(X)$.
\begin{tw}\label{aa}
Assume $f_n\rightrightarrows f$, where all $f_n\in\mathcal B_1^{**}(X,Y)$. The following condition holds:
\begin{equation}\label{star}\tag{$\star$}
\text{if $\varepsilon>0$, the restriction of $f$ to the set $D_f^\varepsilon=\{\xi\in X:\omega_f(\xi)\ge\varepsilon\}$ is continuous}.
\end{equation}
\end{tw}
\begin{proof}
Let $x\in D_f^\varepsilon$. Suppose to the contrary that, for some $\eta>0$, in every $O\in\mathcal T$, $O\ni x$, there is a $\xi\in D_f^\varepsilon$ such that $\varrho(f(x),f(\xi))\ge\eta$. Fix an $n\in\mathbb N$ large enough so that $\varrho_\infty(f,f_n)<\min{\{\eta/4,\varepsilon/2\}}$. Then, if $\xi$ is chosen for a given $O\in\mathcal T$, $O\ni x$, as above, we have
\[
\varrho(f_n(x),f_n(\xi))\ge\varrho(f(x),f(\xi))-\varrho(f_n(\xi),f(\xi))-\varrho(f_n(x),f(x))>\eta-2\cdot\frac\eta4=\frac\eta2,
\]
so $x\in D_{f_n}$. Moreover,
\[
\omega_{f_n}(\xi)>\omega_f(\xi)-2\cdot\frac\varepsilon2\ge0.
\]
It means that all $\xi\in D_{f_n}$ and so $f_n\restriction D_{f_n}$ is discontinuous at $x$, which arrives in contradiction with $f_n\in\mathcal B_1^{**}(X,Y)$.
\end{proof}
For the opposite implication we assume the domain space $X$ is $T_5$ (i.e., hereditarily normal and $T_1$).
\begin{tw}\label{aaa}
Assume $f\colon X\to\mathbb R$ satisfies \eqref{star}. Then there is a sequence $(f_n)_{n=1}^\infty\subset\mathcal B_1^{**}(X)$ such that $f_n\rightrightarrows f$.
\end{tw}
\begin{proof}
Let $\varepsilon>0$. Note that $D_f^\varepsilon$ is a closed subset of $X$. At every $x\in U_f^\varepsilon=X\setminus D_f^\varepsilon$ define $h_1(x)$ and $h_2(x)$ to be correspondingly
$$
\inf_O\sup_{\xi\in O}f(\xi)\qquad\text{and}\qquad\sup_O\inf_{\xi\in O}f(\xi),
$$
where the outer extrema range over all neighbourhoods $O\ni x$. Note that $h_1$ and $h_2$, as functions defined on $U_f^\varepsilon$, are respectively upper and lower semicontinuous. It is easy to check that
$$
f(x)+\varepsilon>f(x)+\omega_f(x)=h_1(x)\ge f(x)\ge h_2(x)=f(x)-\omega_f(x)>f(x)-\varepsilon,
$$
whence $h_2(x)+\varepsilon>f(x)>h_1(x)-\varepsilon$, $x\in U_f^\varepsilon$. As the subspace $U_f^\varepsilon$ of $X$ is $T_4$, from the Kat\v etov--Tong insertion theorem \cite{Katet,Tong}, there exists a continuous function $g\colon U_f^\varepsilon\to\mathbb R$ such that, at every $x\in U_f^\varepsilon$,
$$
f(x)+\varepsilon\ge h_2(x)+\varepsilon\ge g(x)\ge h_1(x)-\varepsilon\ge f(x)-\varepsilon.
$$
So, $|f(x)-g(x)|\le\varepsilon$, $x\in U_f^\varepsilon$. Define $\tilde f$ as $f$ over $D_f^\varepsilon$ and as $g$ over $U_f^\varepsilon$. Obviously, $\varrho_\infty(f,\tilde f)\le\varepsilon$. Note that $\tilde f\in\mathcal B_1^{**}(X)$. Indeed, $D_{\tilde f}\subset D_f^\varepsilon$, so the restriction $\tilde f{\restriction}_{D_{\tilde f}}=f{\restriction}_{D_{\tilde f}}$ is continuous.
\end{proof}
\begin{corl}\label{idfbiwu}
Let $X$ be a $T_5$ topological space and $f\colon X\to\mathbb R$. The following two statements are equivalent:
\begin{itemize}
\item $f$ satisfies \eqref{star};
\item there is a sequence $(f_n)_{n=1}^\infty\subset\mathcal B_1^{**}(X)$ with $f_n\rightrightarrows f$ (abbr.\ $f\in u\mathcal B_1^{**}(X)$).
\end{itemize}
\end{corl}
\begin{prob}
Is $T_5$ essential for Theorem \ref{aaa}?
\end{prob}
The problem boils down to finding the function $g$ (cf.\ the proof). Despite the Kat\v etov--Tong theorem can be applied only if the underlying space $X$ is $T_5$ (and so $U_f^\varepsilon$ is $T_4$), the existence of $g$ can be granted in some particular topological spaces that are far from $T_5$, e.g., in $\mathbb R$ with co-countable topology, which isn't even Hausdorff.

\smallskip
The condition \eqref{star} in terms of \cite{KechLou} means that the {\em oscillation rank} of $f$, $\beta(f)$, equals (at most) one. The iterates of $D^\varepsilon_f$ considered in the next section correspond to larger (still finite) values of $\beta(f)$, namely, $f\in u\mathscr S_n$ iff $\beta(f)\le n$. The oscillation rank has been introduced, with finite and countable values, in order to serve for classifying members of the Baire one class. The range of $\beta(f)$ in the present note is confined to $\mathbb N$, that is, to $f$ living in the {\em first small Baire class} \cite{KechLou}.
\section{Iterates of $D_f^\varepsilon$ and the linear span of $u\mathcal B_1^{**}$}
We proceed with a discussion on sums of functions from $u\mathcal B_1^{**}$. To this end, we need to recall some concepts (although with modified notation) from \cite{MalekIT,Malek} and then extend them. Let $X$ be a topological space, $f\colon X\to\mathbb R$, and $D^0_f$ the closure of set of discontinuity points of $f$. Moreover, let $f^0=f\restriction\domk{D^0_f}$. By induction, given $f^{0,\ldots,0}$ (or briefly $f^{n\textasteriskcentered0}$) and $D^{0,\ldots,0}_f$ (or $D^{n\textasteriskcentered0}_f$), $n\in\mathbb N$ is the number of upper indices here, we set $D^{(n+1)\textasteriskcentered0}_f$ to be the closure of the set of discontinuity points of $f^{n\textasteriskcentered0}\colon D^{n\textasteriskcentered0}_f\to\mathbb R$. We denote with $\mathscr S_n(X)$ (or $\mathscr S_n$ when the domain is understood) the class of all $f$ on $X$ such that $D^{(n+1)\textasteriskcentered0}_f=\emptyset$. Note that (cf.\ page \pageref{vbwlkviyfvwl}) $\mathcal B_1^{**}(X)=\mathscr S_1(X)$.

\smallskip
Now let $\varepsilon>0$. Recall that $D^\varepsilon_f$ is the set of points $x\in X$ such that $\omega_f(x)\ge\varepsilon$, and let $f^\varepsilon$ be $f\restriction D^\varepsilon_f$. Inductively, given arbitrary $\varepsilon_1,\ldots,\varepsilon_n,\varepsilon_{n+1}>0$, $n\in\mathbb N$, we set $D^{\varepsilon_1,\ldots,\varepsilon_n,\varepsilon_{n+1}}_f$ to be the set of points $x\in D^{\varepsilon_1,\ldots,\varepsilon_n}_f$ such that the oscillation of $f^{\varepsilon_1,\ldots,\varepsilon_n}$ (defined on $D_f^{\varepsilon_1,\ldots,\varepsilon_n}$) at $x$, written as $\omega^{\varepsilon_1,\ldots,\varepsilon_n}_f(x)$, is $\ge\varepsilon_{n+1}$; $f^{\varepsilon_1,\ldots,\varepsilon_n,\varepsilon_{n+1}}=f\restriction D^{\varepsilon_1,\ldots,\varepsilon_n,\varepsilon_{n+1}}_f$, $U^{\varepsilon_1,\ldots,\varepsilon_n}_f=X\setminus D^{\varepsilon_1,\ldots,\varepsilon_n}_f$. We will apply the following abbreviation in case all bounds $\varepsilon_1,\ldots,\varepsilon_n$ are equal:
$$D_f^{\overbrace{\scriptscriptstyle\varepsilon,\ldots,\varepsilon}^n}=D_f^{n\textasteriskcentered\varepsilon},\qquad f^{\overbrace{\scriptscriptstyle\varepsilon,\ldots,\varepsilon}^n}=f^{n\textasteriskcentered\varepsilon},\qquad\omega_f^{\overbrace{\scriptscriptstyle\varepsilon,\ldots,\varepsilon}^n}=\omega_f^{n\textasteriskcentered\varepsilon}.$$
The class $u\mathscr S_n(X)$, or just $u\mathscr S_n$, is defined as the class of all functions $f\colon X\to\mathbb R$ such that for arbitrary $\varepsilon_1,\ldots,\varepsilon_{n+1}$, $D^{\varepsilon_1,\ldots,\varepsilon_{n+1}}_f=\emptyset$ (note that it is tantamount to $D^{(n+1)\textasteriskcentered\varepsilon}_f=\emptyset$ for every $\varepsilon>0$). By this definition and by Corollary \ref{idfbiwu}, provided $X$ is $T_5$, $u\mathcal B_1^{**}(X)=u\mathscr S_1(X)$.
\begin{tw}\label{vgebb}
If a sequence $(f_m)_{m=1}^\infty\subset \mathscr S_n(X)$ tends uniformly to an $f$, then $f\in u\mathscr S_n(X)$.
\end{tw}
\begin{proof}
We will proceed by induction on $n$. The case $n=1$ is the content of Theorem \ref{aa}. Assume the claim holds for some $n$ and take a sequence $(f_m)_{m=1}^\infty\subset \mathscr S_{n+1}(X)$, $f_m\rightrightarrows f$. Take an $\varepsilon>0$. For almost all $m$, $\|f-f_m\|_\infty<\varepsilon/2$ and so (for these $m$) $D^\varepsilon_f\subset D_{f_m}$; the latter means $f_m$ is in $\mathscr S_n$ on $D^\varepsilon_f$. By the assumption, as $f_m\rightrightarrows f$ on $D^\varepsilon_f$, $f^\varepsilon\in u\mathscr S_n(D^\varepsilon_f)$, so $D^{(n+1)\textasteriskcentered\varepsilon}_f=\emptyset$.
\end{proof}
\begin{tw}\label{asssa}
Let $X$ be $T_5$, $n\in\mathbb N$. If $f\in u\mathscr S_n(X)$, then there exists a sequence $(f_m)_{m=1}^\infty\subset\mathscr S_n(X)$ such that $f_m\rightrightarrows f$.
\end{tw}
\begin{proof}
Again induction on $n$. The result holds for $n=1$ (Theorem \ref{aaa}). Assume the claim is fulfilled for some $n$ and consider an $f\in u\mathscr S_{n+1}(X)$ and $\varepsilon>0$. By definition, $f^\varepsilon\in u\mathscr S_n(D_f^\varepsilon)$. Apply the assumption to $f^\varepsilon$ and pick a function $g\in\mathscr S_n(D_f^\varepsilon)$ with $\|f^\varepsilon-g\|_\infty<\varepsilon$. Moreover, as $\omega_f(x)<\varepsilon$ at every $x\in U^\varepsilon_f$, arguing like in the proof of Theorem \ref{aaa}, we can arrive to a continuous function $h\colon U^\varepsilon_f\to\mathbb R$ such that (over $U^\varepsilon_f$) $\|f-h\|_\infty<\varepsilon$. So, $g$ and $h$ taken jointly form a function $\tilde f$ on $X$ satisfying $\|f-\tilde f\|_\infty<\varepsilon$. Moreover, $\tilde f\in\mathscr S_{n+1}(X)$. Indeed, $\domk D_{\tilde f}\subset D^\varepsilon_f$ and so $\tilde f\restriction \domk D_{\tilde f}\subset g\in\mathscr S_n(D^\varepsilon_f)$. Thus $\tilde f\restriction\domk D_{\tilde f}\in\mathscr S_n$. The proof is over.
\end{proof}
\begin{tw}\label{a2sabis}
Let $X$ be $T_5$, $f\in u\mathscr S_n(X)$, and $g\in u\mathscr S_m(X)$. Then $f+g\in u\mathscr S_{n+m}(X)$.
\end{tw}
\begin{proof}
By Theorem \ref{asssa} there are sequences $(f_s)_{s=1}^\infty\subset\mathscr S_n(X)$, $(g_s)_{s=1}^\infty\subset\mathscr S_m(X)$ such that $f_s\rightrightarrows f$ and $g_s\rightrightarrows g$.
By \cite{MalekIT}, $f_s+g_s\in\mathscr S_{n+m}(X)$ for every $s$. Since $f_s+g_s\rightrightarrows f+g$, from Theorem \ref{vgebb}, $f+g\in u\mathscr S_{n+m}(X)$.
\end{proof}
As it is shown in \cite{MalekIT},
$\mathscr S_n+\mathscr S_m=\mathscr S_{n+m}$ over any topological space. In particular, every function $f\in\mathscr S_n$ can be written as the sum of at most $n$ functions from $\mathcal B^{**}_1=\mathscr S_1$. We were unable to prove the analogous result for $u\mathscr S_n$, however we show that every $f\in\mathscr S_n$ can be decomposed into finitely many terms from $u\mathscr S_1$ (Theorem \ref{vgebvuyb}).

We will say a space $X$ has the property (T) if the following assertion, more general than the Tietze extension theorem is true: {\em Let $Y\subset X$ be a closed subspace, $f\colon Y\to\mathbb R$ an arbitrary function. Then there exists a continuous function $g\colon X\setminus Y\to\mathbb R$ such that if $h=f\cup g$, $\omega_f(x)=\omega_h(x)$ at every $x\in Y$.} Clearly, every space with the property (T) must be $T_4$. It is easy to see that, e.g., $\mathbb R$ with the natural topology has this property.
\begin{prob}
Does every $T_4$ topological space have property {\rm(T)}?
\end{prob}
\begin{lm}\label{jhcffc}
Let $X$ be a $T_6$ space. Then every open subset $O\subset X$ can be written as a union $\bigcup_{k=1}^\infty E_k$ of closed sets $E_k$ such that $E_k\subset\wne E_{k+1}$ for every $k\in\mathbb N$.
\end{lm}
\begin{proof}
By the $T_6$ property of $X$ we have $O=\bigcup_{k=1}^\infty D_k$, where all $D_k$'s are closed. Set $E_1=D_1$. Assume $E_k\supset D_k$ is defined and take two open and disjoint sets $U,V$ with $E_k\cup D_{k+1}\subset U$ and $X\setminus O\subset V$. Define $E_{k+1}=\domk U$. Clearly, $D_k\subset E_k\subset U\subset\wne\domk U=\wne E_{k+1}\subset E_{k+1}\subset X\setminus V\subset O$ and so $O=\bigcup_{k=1}^\infty D_k\subset\bigcup_{k=1}^\infty E_k\subset O$.
\end{proof}
\begin{tw}\label{vgebvuyb}
Let $f\in u\mathscr S_{n+1}(X)$, $n\in\mathbb N$, $X$ be a $T_6$ space with property {\rm(T)}. Then there exist functions $f_1,f_3\in u\mathscr S_n(X)$, $f_2\in u\mathscr S_1(X)$, such that $f=f_1+f_2+f_3$.
\end{tw}
\begin{proof}
Set $D=D_f^1$. Then both restrictions $f\restriction D$ and $f\restriction(X\setminus D)$ are in $u\mathscr S_n$.

The term $f_1$ is defined as follows: $f_1=f$ on $D$, while on $X\setminus D$ it is extended continuously so that $\omega_{f_1}(x)=\omega^1_f(x)$ at every $x\in D$ (property~(T)). Clearly, $f_1\in u\mathscr S_n(X)$. Indeed, if $\varepsilon>0$, then $D_{f_1}^\varepsilon\subset D$ and so $D_{f_1}^\varepsilon=D_f^{1,\varepsilon}$; hence the restriction of $f_1$ to $D_{f_1}^{n\textasteriskcentered\varepsilon}=D_f^{1,n\textasteriskcentered\varepsilon}$ (if nonempty) is continuous.

Write down $X\setminus D=\bigcup_{k=1}^\infty E_k$, where the sequence $(E_k)_{k=1}^\infty$ of closed sets is as in Lemma \ref{jhcffc}. Note that the restriction of $f-f_1$ (like of $f$ alone) to the set $X\setminus D$ is in $u\mathscr S_n(X\setminus D)$. Thus, $(f-f_1)\restriction(D_f^{n\textasteriskcentered1}\cap E_1)=g_1$ is continuous. Moreover, for the same reason, $(f-f_1)\restriction\bigl(D_f^{n\textasteriskcentered1}\cap\wne E_1\cup D_f^{n\textasteriskcentered1/2}\cap(E_2\setminus\wne E_1)\bigr)=g_2$ is a continuous function defined on a closed subset of $D_f^{n\textasteriskcentered1/2}\cap E_2$. And so on by induction---note that the restriction of $f-f_1$ to the (closed) set
$$
\bigcup_{s=0}^k\,\bigl(D_f^{n\textasteriskcentered1/(s+1)}\cap(E_{s+1}\setminus\wne E_s)\bigr)\subset D_f^{n\textasteriskcentered1/(k+1)}\cap E_{k+1},\qquad k\in\mathbb N,
$$
(here $E_1\setminus\wne E_0$ we accept as $\wne E_1$), called $g_{k+1}$, is continuous. Having in mind the property of $(E_k)_{k=1}^\infty$ described in Lemma \ref{jhcffc}, one can claim that also the limit $\lim_{k\to\infty}g_k$, which is the restriction of $f-f_1$ to the set
\begin{equation}
\bigcup_{s=0}^\infty\,\bigl(D_f^{n\textasteriskcentered1/(s+1)}\cap(E_{s+1}\setminus\wne E_s)\bigr),\label{woieu}
\end{equation}
is a continuous function and its domain, \eqref{woieu}, is a closed subset of $X\setminus D$. Thus, from the normality of $X\setminus D$, there is a continuous extension $g$ of $\lim_{k\to\infty}g_k$ defined on $X\setminus D$. Set $f_2$ as $g$ on $X\setminus D$ and as $0$ on $D$. Clearly, $f_2\in\mathcal B_1^{**}\subset u\mathscr S_1$.

\smallskip
It remains to check that $f_3=f-f_1-f_2$ is a $u\mathscr S_n$ function. Indeed, for arbitrary positive $\varepsilon<1$ and $i$ (as $f_1$ and $f_2$ were continuous in $X\setminus D$), $D_f^{i\textasteriskcentered\varepsilon}\setminus D=D_{f_3}^{i\textasteriskcentered\varepsilon}\setminus D$. Hence $D_{f_3}^{n\textasteriskcentered\varepsilon}\setminus D=\emptyset$ (since $f$ is $u\mathscr S_n$ on $X\setminus D$). Take $k\in\mathbb N$ such that $\varepsilon>1/k$. By definition, $f_3(x)=0$ at all $x\in D$ and also at every $x\in D_f^{n\textasteriskcentered1/k}\setminus E_{k-1}\subset D_f^{n\textasteriskcentered\varepsilon}\setminus E_{k-1}$. As $E_{k-1}$ is a closed set, $f_3^{n\textasteriskcentered\varepsilon}$ is continuous at every $x$ (from its domain) that lies in $D$. Concluding, $f_3^{n\textasteriskcentered\varepsilon}$ is a continuous function and so $f_3\in u\mathscr S_n$.
\end{proof}
By induction we obtain
\begin{wn}
Let $f$ and $X$ be as in the theorem above. Then $f$ can be written as the sum of at most $2^{n+1}-1$ many terms from $u\mathscr S_1(X)$.
\end{wn}
It may be of some interest to have the bound $2^{n+1}-1$ above improved. It might be well even $n+1$.

\end{document}